\documentclass[11pt]{amsart}
\usepackage{amsmath,amssymb}
\usepackage[dvips]{graphicx}
\usepackage{psfrag}
\begin{document}

%
\newtheorem{theorem}{Theorem}
\newtheorem{algorithm}[theorem]{Algorithm}
\newtheorem{conjecture}[theorem]{Conjecture}
\newtheorem{axiom}[theorem]{Axiom}
\newtheorem{corollary}[theorem]{Corollary}
\newtheorem{definition}[theorem]{Definition}
\newtheorem{fact}[theorem]{Fact}
\newtheorem{lemma}[theorem]{Lemma}
\newtheorem{proposition}[theorem]{Proposition}
\newtheorem*{claim}{Claim}
%
\newcommand{\OMIT}[1]{}
\def\N{\mathbb{N}}
\def\R{\mathbb{R}}
\def\P{\mathcal{P}}
\def\C{\mathcal{C}}
\def\U{\mathcal{U}}
\def\E{\mathbb{E}}
\def\eul{\rm{e}}
\newcommand{\norm}[1]{\left\|{#1}\right\|}
\newcommand{\pde}[2]{\frac{\partial #1}{\partial #2}}
\newcommand{\fundef}[3]{#1:#2\to #3}
\newcommand{\abs}[1]{\left|{#1}\right|}
\newcommand{\defeq}{\stackrel{\triangle}{=}}
\renewcommand{\Gamma}{\varGamma}
\renewcommand{\epsilon}{\varepsilon}
\newcommand{\eps}{\varepsilon}
\renewcommand{\bar}{\overline}
\renewcommand{\hat}{\widehat}
\newcommand{\ie}{i.e.\@ }
\newcommand{\etc}{etc.\@ }
\newcommand{\wrt}{w.r.t.\@ }
\newcommand{\absfrac}{\!/\!\:}
%
\def\COMMENT#1{}
\def\TASK#1{}
\def\noproof{{\unskip\nobreak\hfill\penalty50\hskip2em\hbox{}\nobreak\hfill%
       $\square$\parfillskip=0pt\finalhyphendemerits=0\par}\goodbreak}
\def\endproof{\noproof\bigskip}
\newdimen\margin   
\def\textno#1&#2\par{%
   \margin=\hsize
   \advance\margin by -4\parindent
          \setbox1=\hbox{\sl#1}%
   \ifdim\wd1 < \margin
      $$\box1\eqno#2$$%
   \else
      \bigbreak
      \hbox to \hsize{\indent$\vcenter{\advance\hsize by -3\parindent
      \sl\noindent#1}\hfil#2$}%
      \bigbreak
   \fi}
\def\proof{\removelastskip\penalty55\medskip\noindent{\bf Proof. }}
\def\enddiscard{}
\long\def\discard#1\enddiscard{}

\date{} 
\title{A Dirac type result on Hamilton cycles in oriented graphs}
\author{Luke Kelly, Daniela K\"{u}hn \and Deryk Osthus}
\thanks {D.~K\"uhn was partially supported by the EPSRC, grant no.~EP/F008406/1.
D.~Osthus was partially supported by the EPSRC, grant no.~EP/E02162X/1 and~EP/F008406/1.}
\begin{abstract}
We show that for each $\alpha>0$ every sufficiently large oriented graph~$G$ with
$\delta^+(G),\delta^-(G)\ge 3|G|/8+ \alpha |G|$ contains a Hamilton cycle.
This gives an approximate solution to a problem of Thomassen~\cite{thomassen_81_long_cycles}.
In fact, we prove the stronger result that~$G$ is still Hamiltonian if
$\delta(G)+\delta^+(G)+\delta^-(G)\geq 3|G|/2 + \alpha |G|$.
Up to the term~$\alpha |G|$ this confirms a conjecture of H\"aggkvist~\cite{HaggkvistHamilton}.
We also prove an Ore-type theorem for oriented graphs.
\end{abstract}
\maketitle
%
\section{Introduction}\label{sec:introduction}
An \emph{oriented graph~$G$} is obtained from
a (simple) graph by orienting its edges. Thus between every pair of vertices of~$G$ there
exists at most one edge. The \emph{minimum semi-degree~$\delta^0(G)$} of~$G$ is the minimum of its
minimum outdegree~$\delta^+(G)$ and its minimum indegree~$\delta^-(G)$.
When referring to paths and cycles in oriented graphs we always mean that these are directed without
mentioning this explicitly.

A fundamental result of Dirac states that a minimum degree of~$|G|/2$ guarantees a Hamilton cycle in an
undirected graph~$G$. There is an analogue of this for digraphs due to
Ghouila-Houri~\cite{GhouilaHouri} which states that every digraph~$D$ with minimum
semi-degree at least~$|D|/2$ contains a Hamilton cycle. The bounds on
the minimum degree in both results are best possible. A natural question is to ask for the
(smallest) minimum semi-degree which guarantees a Hamilton cycle in an oriented graph~$G$.
This question was first raised by Thomassen~\cite{thomassen_79_long_cycles_constraints},
who~\cite{thomassen_82} showed that a minimum semi-degree of $|G|/2-\sqrt{|G|/1000}$ suffices
(see also~\cite{thomassen_81_long_cycles}). Note that this
degree requirement means that~$G$ is not far from being a tournament.
H\"aggkvist~\cite{HaggkvistHamilton} improved the bound further to $|G|/2-2^{-15}|G|$
and conjectured that the actual value lies close to~$3|G|/8$.
The best previously known bound is due to H\"aggkvist and Thomason~\cite{HaggkvistThomasonHamilton},
who showed that for each~$\alpha>0$ every sufficiently large oriented graph~$G$
with minimum semi-degree at least $(5/12 + \alpha) |G|$ has a Hamilton cycle.
Our first result implies that the actual value is indeed close to~$3|G|/8$.

\begin{theorem}\label{thm:minsemi}
For every $\alpha>0$ there exists an integer $N=N(\alpha)$ such that every oriented
graph~$G$ of order $|G|\geq N$ with $\delta^0(G) \geq (3/8+\alpha)|G|$ contains a Hamilton cycle.
\end{theorem}
A construction of H\"aggkvist~\cite{HaggkvistHamilton} shows that the bound in Theorem~\ref{thm:minsemi}
is essentially best possible (see Proposition \ref{proposition:extremal_example}).

In fact, H\"aggkvist~\cite{HaggkvistHamilton} formulated the following stronger conjecture.
Given an oriented graph~$G$, let~$\delta(G)$ denote the minimum degree of~$G$ 
(i.e.~the minimum number of edges incident to a vertex) and set
$\delta^*(G):=\delta(G)+\delta^+(G)+\delta^-(G)$.
\begin{conjecture}[H\"aggkvist~\cite{HaggkvistHamilton}]\label{conj:Haegg}
Every oriented graph~$G$ with $\delta^*(G)>(3n-3)/2$ has a Hamilton cycle.
\end{conjecture}
Our next result provides an approximate confirmation of this conjecture for large oriented graphs. 
\begin{theorem}\label{theorem:main}
For every $\alpha>0$ there exists an integer $N=N(\alpha)$ such that every oriented
graph~$G$ of order $|G|\geq N$ with $\delta^*(G) \geq (3/2+\alpha)|G|$ contains a Hamilton cycle.
\end{theorem}
Note that Theorem~\ref{thm:minsemi} is an immediate consequence of this.
The proof of Theorem~\ref{theorem:main} can be modified to yield the following Ore-type analogue
of Theorem~\ref{thm:minsemi}. (Ore's theorem~\cite{ore} states that every graph $G$ on $n \ge 3$ 
vertices which satisfies
$d(x)+d(y) \ge n$ whenever $xy \notin E(G)$ has a Hamilton cycle.)
\begin{theorem}\label{thm:Ore}
For every $\alpha>0$ there exists an integer $N=N(\alpha)$ such that every oriented
graph~$G$ of order $|G|\geq N$ with $d^+(x)+d^-(y)\ge (3/4+\alpha)|G|$ whenever $xy\notin E(G)$
contains a Hamilton cycle.
\end{theorem}
A version for general digraphs was proved by Woodall~\cite{woodall}:
every strongly connected digraph $D$ on $n \ge 2$ vertices which satisfies
$d^+(x)+d^-(y) \ge n$ whenever $xy \notin E(D)$ has a Hamilton cycle.

Theorem~\ref{thm:minsemi} immediately implies a partial result towards a
conjecture of Kelly (see e.g.~\cite{digraphsbook}), which states that
every regular tournament on~$n$ vertices can be
partitioned into~$(n-1)/2$ edge-disjoint Hamilton cycles. (A regular tournament is an orientation
of a complete graph in which the indegree of every vertex equals its outdegree.)
\begin{corollary}\label{corollary:kellys_conjecture_result}
For every $\alpha>0$ there exists an integer $N=N(\alpha)$ such that every regular
tournament of order~$n\ge N$ contains at least $(1/8-\alpha)n$ edge-disjoint Hamilton cycles.
\end{corollary}
Indeed, Corollary~\ref{corollary:kellys_conjecture_result} follows from
Theorem~\ref{thm:minsemi} by successively removing
Hamilton cycles until the oriented graph~$G$ obtained from the tournament in
this way has minimum semi-degree less than $(3/8+\alpha)|G|$. The best
previously known bound on the number of edge-disjoint Hamilton cycles in a
regular tournament is the one which follows from the result of
H\"aggkvist and Thomason~\cite{HaggkvistThomasonHamilton} mentioned above.
A related result of Frieze and Krivelevich~\cite{FKhampack} states that 
every dense $\eps$-regular digraph contains a collection of edge-disjoint Hamilton cycles which covers
almost all of its edges. This immediately implies that the same holds for almost every 
tournament. Together with a lower bound by McKay~\cite{McKay} on the number of 
regular tournaments, the above result in~\cite{FKhampack} also implies that almost every regular tournament
contains a collection of edge-disjoint Hamilton cycles which covers
almost all of its edges.%
\COMMENT{The result of McKay implies that the number of regular tournaments is at least
$2^{\binom{n}{2}}/n^n$. So the probability that a tournament is regular is at least
$n^{-n}$. The definition of dense $\eps$-regular in Frieze and Krivelevich is
(P1) $\delta^0 \ge \alpha n$ and (P2) $|S|, |T| \ge \eps n$ implies $|d(S,T) - \alpha | \le \eps$.
Then for a random tournament, a Chernoff bound implies the probability that P2 fails (with ($\alpha=1/2-\eps/2$) 
for a fixed pair $S,T$ with $|S|, |T| \ge \eps n$ 
is at most $e^{-c n^2}$, where $c$ depends only on $\eps$.
So for a random tournament,the probability that P2 fails  
(for some  pair $S,T$ with $|S|, |T| \ge \eps n$) is at most $2^{2n} e^{-c n^2} \le e^{-c n^2/2}$.
Let $A$ be the event that a random tournament is not $\eps$-regular.
Let $B$ be the event that it is regular in the classical sense.
Then altogether, we have
$P (A \mid B)= \frac{P( A \cap B)}{P(B)} = \frac{ P(\mbox{P2 fails}  \cap B)}{P(B)}
\le \frac{ P(\mbox{P2 fails} )}{P(B)} \le n^n  e^{-c n^2/2} \to 0$.}

Note that Theorem~\ref{theorem:main} implies that for sufficiently
large tournaments~$T$ a minimum semi-degree of at least~$(1/4+\alpha)|T|$ already
suffices to guarantee a Hamilton cycle. (However, it is not hard to prove this
directly.)%
     \COMMENT{A minimum semi-degree of at least~$(1/4+\alpha)|T|$ forces~$T$
to be strong and then we can use Moon's thm that a tournament is hamiltonian
iff it is strong.}
It was shown by Bollob\'as and H\"aggkvist~\cite{BHpower}
that this degree condition even ensures the $k$th power of a Hamilton
cycle (if~$T$ is sufficiently large compared to~$1/\alpha$ and~$k$).
The degree condition is essentially best possible as a minimum semi-degree
of $|T|/4-1$ does not even guarantee a single Hamilton cycle.%
     \COMMENT{Let $|T|=4m+2=|A|+|B|$ where $|A|=|B|=2m+1$ and $A$ and $B$
regular tournaments. Add all the $A$-$B$ edges to obtain~$T$.
Then $\delta^0(T)=m=(|T|-2)/4>|T|/4-1$.}

Since this paper was written, we have used some of the tools and methods to obtain
an exact version of Theorem~\ref{thm:minsemi} (but not of Theorems~\ref{theorem:main} and~\ref{thm:Ore})
for large oriented graphs~\cite{hamexact} as well as
an approximate analogue of Chv\'atal's theorem on Hamiltonian degree sequences for digraphs~\cite{KOT}.
See~\cite{shortcycles} for related results about short cycles and pancyclicity for oriented graphs.

Our paper is organized as follows. In the next section we introduce some basic definitions and describe the
extremal example which shows that Theorem~\ref{thm:minsemi} (and thus also
Theorems~\ref{theorem:main} and~\ref{thm:Ore}) is essentially best possible. Our proof of
Theorem~\ref{theorem:main} relies on the Regularity lemma for digraphs and on a variant
(due to Csaba~\cite{csaba_06_bollobas_eldridge}) of the Blow-up lemma.
These and other tools are introduced in~Section~\ref{sec:regularity_lemma},
where we also give an overview of the proof. In Section~\ref{sec:shifted_paths} we
collect some preliminary results. Theorem \ref{theorem:main} is then proved in
Section~\ref{sec:main_proof}. In the last section we discuss the modifications
needed to prove Theorem~\ref{thm:Ore}.
%
%
%
\section{Notation and the extremal example}\label{sec:notation_extremal_examples}
Before we show that Theorems~\ref{thm:minsemi}, \ref{theorem:main} and~\ref{thm:Ore} are
essentially best possible, we will introduce
the basic notation used throughout the paper. Given two vertices~$x$ and~$y$ of
an oriented graph~$G$, we write~$xy$ for the edge directed from~$x$ to~$y$.
The order~$|G|$ of~$G$ is the number of its vertices.
We write~$N^+_G(x)$ for the outneighbourhood of a vertex~$x$ and $d^+_G(x):=|N^+_G(x)|$ for its outdegree.
Similarly, we write~$N^-_G(x)$ for the inneighbourhood of~$x$ and $d^-_G(x):=|N^-_G(x)|$ for its indegree.
We write $N_G(x):=N^+_G(x)\cup N^-_G(x)$ for the neighbourhood of~$x$ and use~$N^+(x)$ etc.~whenever
this is unambiguous. We write~$\Delta(G)$ for the maximum of $|N(x)|$ over all vertices $x\in G$.

Given a set~$A$ of vertices of~$G$, we write $N^+_G(A)$ for the set of all outneighbours of vertices in~$A$.
So~$N^+_G(A)$ is the union of $N^+_G(a)$ over all $a\in A$. $N^-_G(A)$ is defined similarly.
The oriented subgraph of~$G$ induced by~$A$ is denoted by~$G[A]$.
Given two vertices $x,y$ of~$G$, an \emph{$x$-$y$ path} is a directed path which joins~$x$ to~$y$.
Given two disjoint subsets~$A$ and~$B$ of vertices of~$G$, an~$A$-$B$ edge is an edge~$ab$
where $a\in A$ and $b\in B$, the set of these edges is denoted by $E_G(A,B)$ and we
put $e_G(A,B):=|E_G(A,B)|$.

Recall that when referring to paths and cycles in oriented graphs we always mean that
they are directed without mentioning this explicitly. Given two vertices~$x$ and~$y$ on a
directed cycle~$C$, we write $xCy$ for the subpath of~$C$ from $x$ to~$y$.
Similarly, given two vertices~$x$ and~$y$ on a directed path~$P$
such that~$x$ precedes~$y$, we write $xPy$ for the subpath of~$P$ from $x$ to~$y$.
A \emph{walk} in an oriented graph~$G$ is a sequence of (not necessary distinct)
vertices $v_1,v_2,\ldots,v_{\ell}$ where $v_iv_{i+1}$ is an edge for all $1\leq i<\ell$.
The walk is \emph{closed} if $v_1=v_\ell$. A \emph{$1$-factor} of~$G$ is a collection
of disjoint cycles which cover all the vertices of~$G$.
We define things similarly for graphs and for directed graphs.
The \emph{underlying graph} of an oriented graph~$G$ is the graph obtained from~$G$
by ignoring the directions of its edges.

Given disjoint vertex sets~$A$ and~$B$ in a graph~$G$, we write $(A,B)_G$
for the induced bipartite subgraph of~$G$ whose vertex classes are~$A$ and~$B$.
We write~$(A,B)$ where this is unambiguous.
We call an orientation of a complete graph a \emph{tournament} and an orientation of a complete bipartite
graph a \emph{bipartite tournament}.
An oriented graph~$G$ is \emph{d-regular} if all vertices have in- and outdegree~$d$. $G$ is
\emph{regular} if it is $d$-regular for some~$d$. It is easy to see (e.g.~by induction)%
     \COMMENT{Start with a triangle. Always add two vertices $a$ and $b$ to the old vertex set.
Divide the old vertex set into 2 sets $A$ and $B$ such that $|A|=|B|+1$. Add all the
$A$-$a$ edges, all the $a$-$B$ edges, all the $b$-$A$ edges as well as all the $B$-$b$ edges.
Finally, add $ab$.}  
that for every odd~$n$ there exists a regular tournament on~$n$ vertices. 
Throughout the paper we omit floors and ceilings whenever this does not affect the
argument.

The following construction by H\"aggkvist~\cite{HaggkvistHamilton} shows that
Conjecture~\ref{conj:Haegg} is best possible for infinitely many values
of~$|G|$.%
    \COMMENT{H\"aggkvist describes his 1-factor result as `more or less sharp'}
We include it here for completeness.

\begin{proposition}\label{proposition:extremal_example}
There are infinitely many oriented graphs~$G$ with minimum semi-degree~$(3|G|-5)/8$
which do not contain a $1$-factor and thus do not contain a Hamilton cycle. 
\end{proposition}
\begin{proof}
Let $n:=4m+3$ for some odd $m\in\N$. Let~$G$ be the oriented graph obtained from the disjoint
union of two regular tournaments~$A$ and~$C$ on~$m$ vertices, a set~$B$ of~$m+2$ vertices
and a set~$D$ of~$m+1$ vertices by adding all edges from~$A$ to~$B$, all edges from~$B$ to~$C$, all edges
from~$C$ to~$D$ as well as all edges from~$D$ to~$A$. Finally, between~$B$ and~$D$ we add edges to
obtain a bipartite tournament which is as regular as possible, i.e.~the in- and outdegree of every
vertex differ by at most~1. So in particular every vertex in~$B$ sends exactly
$(m+1)/2$ edges to~$D$ (Figure~1).%
    \COMMENT{To see that such a tournament exists, pick a vertex $b\in B$ and divide~$D$
into 2 sets $D_1$ and $D_2$ of equal size. Include a regular bipartite tournament
between $B\setminus \{b\}$ and $D$. (To see that this exists, take a blow-up of a 4-cycle.)
The outneighbourhood of $b$ will be $D_1$ and its inneighbourhood will be $D_2$.}

It is easy to check that the minimum semi-degree of~$G$ is $(m-1)/2+(m+1)=(3n-5)/8$, as required.%
     \COMMENT{All vertices $x\in A$ have 
$d^-(x)=\frac{1}{2}(m-1)+(m+1)=\frac{3}{2}\left(\frac{n}{4}-\frac{3}{4}\right)+\frac{1}{2}
=\frac{3}{8}n-\frac{5}{8}$. For the vertices $x\in B$ have $d^-(x)= \frac{1}{2}(m+1)+m=\frac{1}{2}(m-1)+(m+1)
=\frac{3}{8}n-\frac{5}{8}$. Finally, for the vertices in~$D$ have $d^-(x)\ge \frac{1}{2}(m+1)+m
=\frac{3}{8}n-\frac{5}{8}$.} 
Since every path which joins two vertices in~$B$ has to pass through~$D$, it follows that
every cycle contains at least as many vertices from~$D$ as it contains from~$B$.
As $|B|>|D|$ this means that one cannot cover all the vertices of~$G$ by disjoint cycles,
i.e.~$G$ does not contain a 1-factor. 
\end{proof}
\begin{figure}\label{fig:extremal}
\centering\footnotesize
\psfrag{1}[][]{$A$}
\psfrag{2}[][]{$B$}
\psfrag{3}[][]{$C$}
\psfrag{4}[][]{$D$}
\includegraphics[scale=0.8]{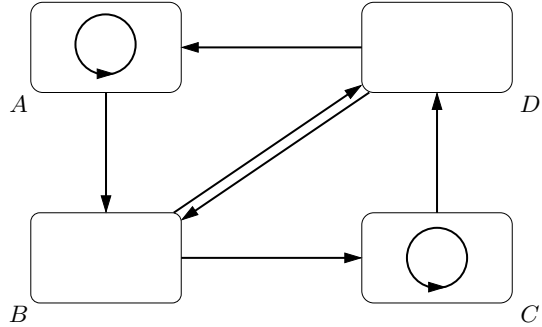}
\caption{The oriented graph in the proof of Proposition~\ref{proposition:extremal_example}.}
\end{figure}
%
%
%
\section{The Diregularity lemma, the Blow-up lemma and other tools}\label{sec:regularity_lemma}
\subsection{The Diregularity lemma and the Blow-up lemma}
In this section we collect all the information we need about the
Diregularity lemma and the Blow-up lemma. See~\cite{KSi} for a survey on the Regularity lemma
and~\cite{JKblowup} for a survey on the Blow-up lemma.
We start with some more notation. The density of a bipartite graph $G = (A,B)$
with vertex classes~$A$ and~$B$ is defined to be
\[d_G(A,B) := \frac{e_G(A,B)}{\abs{A}\abs{B}}.\]
We often write $d(A,B)$ if this is unambiguous. Given $\epsilon>0$, we say
that~$G$ is \emph{$\epsilon$-regular} if for all subsets $X\subseteq A$ and $Y\subseteq B$
with $\abs{X}>\epsilon\abs{A}$ and $\abs{Y}>\epsilon\abs{B}$ we have that
$\abs{d(X,Y)-d(A,B)}<\epsilon$. Given $d\in [0,1]$ we say that~$G$ is $(\epsilon,d)$-\emph{super-regular}
if it is $\eps$-regular and furthermore $d_G(a)\ge (d-\eps) \abs{B}$ for all $a\in A$ and
$d_G(b)\ge (d-\eps)\abs{A}$ for all $b\in B$. (This is a slight variation of the standard definition
of $(\epsilon,d)$-super-regularity where one requires $d_G(a)\ge d \abs{B}$ and
$d_G(b)\ge d\abs{A}$.) 

The Diregularity lemma is a version of the Regularity lemma for digraphs due to
Alon and Shapira~\cite{AlonShapiraTestingDigraphs}. Its proof is quite similar to
the undirected version.
We will use the degree form of the Diregularity lemma which can be easily derived
(see e.g.~\cite{young_05_extremal}) from the standard version, in exactly the same manner
as the undirected degree form.
\begin{lemma}[Degree form of the Diregularity lemma]\label{lemma:diregularity_lemma}
For every $\epsilon\in (0,1)$ and every integer~$M'$ there are integers~$M$ and~$n_0$
such that if~$G$ is a digraph on $n\ge n_0$ vertices and
$d\in[0,1]$ is any real number, then there is a partition of the vertices of~$G$ into
$V_0,V_1,\ldots,V_k$, a spanning subdigraph~$G'$ of~$G$ and a set~$U$ of ordered pairs
$V_iV_j$ (where $1\le i,j\le k$ and $i\neq j$) such that the following holds:
\begin{itemize}
\item $M'\le k\leq M$,
\item $\abs{V_0}\leq \epsilon n$,
\item $\abs{V_1}=\cdots=\abs{V_k}=:m$,
\item $d^+_{G'}(x)>d^+_G(x)-(d+\epsilon)n$ for all vertices $x\in G$,
\item $d^-_{G'}(x)>d^-_G(x)-(d+\epsilon)n$ for all vertices $x\in G$,
\item $|U|\le \eps k^2$,
\item for every ordered pair $V_iV_j\notin U$ with $1\le i,j\le k$ and $i\neq j$ the
bipartite graph $(V_i,V_j)_G$ whose vertex classes are~$V_i$ and~$V_j$ and whose edge set is the
set $E_G(V_i,V_j)$ of all the $V_i$-$V_j$ edges in~$G$ is $\epsilon$-regular,
\item $G'$ is obtained from $G$ by deleting the following edges of~$G$: all edges
with both endvertices in~$V_i$ for all $i\ge 1$ as well as all edges in $E_G(V_i,V_j)$ for
all $V_iV_j\in U$ and for all those $V_iV_j\notin U$ with $1\le i,j\le k$ and $i\neq j$
for which the density of $(V_i,V_j)_G$ is less than~$d$.
\end{itemize}
\end{lemma}

$V_1,\ldots,V_k$ are called \emph{clusters}, $V_0$ is called the \emph{exceptional set}
and the vertices in~$V_0$ are called \emph{exceptional vertices}. $U$ is called the set of
\emph{exceptional pairs of clusters}.
Note that the last two conditions of the lemma imply that for all $1\leq i,j\leq k$ with $i\neq j$
the bipartite graph~$(V_i,V_j)_{G'}$ is $\epsilon$-regular and has density either~$0$ or density at least~$d$.
In particular, in~$G'$ all pairs of clusters are $\epsilon$-regular in both
directions (but possibly with different densities).
We call the spanning digraph~$G'\subseteq G$ given by the Diregularity lemma the \emph{pure digraph}.
Given clusters $V_1,\ldots,V_k$ and the pure digraph~$G'$, the \emph{reduced digraph~$R'$} is the digraph
whose vertices are $V_1,\ldots,V_k$ and in which $V_iV_j$ is an edge if and only if~$G'$ contains a $V_i$-$V_j$ edge.
Note that the latter holds if and only if $(V_i,V_j)_{G'}$
is $\epsilon$-regular and has density at least~$d$.
It turns out that~$R'$ inherits many properties of~$G$, a fact that is crucial in our proof.
However,~$R'$ is not necessarily oriented even if the original digraph~$G$ is, but the next lemma
shows that by discarding edges with appropriate probabilities one can go over to a
reduced oriented graph $R\subseteq R'$ which still inherits many of the properties of~$G$.
(d) will only be used in the proof of Theorem~\ref{thm:Ore}.
\begin{lemma}\label{lemma:reduced_oriented}
For every $\epsilon\in (0,1)$ there exist integers~$M'=M'(\epsilon)$ and $n_0=n_0(\epsilon)$
such that the following holds. Let $d \in [ 0,1 ]$ and let~$G$ be an oriented graph of order at
least~$n_0$ and let~$R'$ be the reduced digraph and~$U$ the set of exceptional pairs of clusters
obtained by applying the Diregularity lemma to~$G$ with parameters $\epsilon$, $d$ and~$M'$.
Then~$R'$ has a spanning oriented subgraph~$R$ with 
\begin{itemize}
\item[{\rm (a)}]$\delta^+(R)\ge (\delta^+(G)/|G|-(3\epsilon+d))\abs{R}$,
\item[{\rm (b)}] $\delta^-(R)\ge(\delta^-(G)/|G|-(3\epsilon+d))\abs{R}$,
\item[{\rm (c)}] $\delta(R)\ge(\delta(G)/|G|-(3\epsilon+2d))\abs{R}$,
\item[{\rm (d)}] if $2\eps\le d\le 1-2\eps$ and $c\ge 0$ is such that $d^+(x)+d^-(y)\ge c|G|$ whenever $xy\notin E(G)$
then $d^+_R(V_i)+d^-_R(V_j)\ge (c-6\epsilon -2d)|R|$ whenever $V_iV_j\notin E(R)\cup U$.
\end{itemize}
\end{lemma}
\begin{proof}
Let us first show that every cluster~$V_i$ satisfies
\begin{equation}\label{eq:NVi} 
|N_{R'}(V_i)|/|R'|\ge \delta(G)/|G|-(3\epsilon+2d).
\end{equation}
To see this, consider any vertex $x\in V_i$. As~$G$ is an oriented graph, the Diregularity
lemma implies that
$|N_{G'}(x)|\ge \delta(G)-2(d+\eps)|G|$. On the other hand, 
$|N_{G'}(x)|\le |N_{R'}(V_i)|m +|V_0|\le |N_{R'}(V_i)||G|/|R'|+\eps |G|$.
Altogether this proves~(\ref{eq:NVi}).

We first consider the case when
\begin{equation}\label{eq:notsmall}
\delta^+(G)/|G|\ge 3\eps+d \text{\ \ \ and \ \ \ } \delta^-(G)/|G|\ge 3\eps+d
\text{\ \ \ and \ \ \ } c\ge 6\eps +2d.
\end{equation}
Let~$R$ be the spanning oriented subgraph obtained from~$R'$ by deleting edges randomly as follows.
For every unordered pair~$V_i,V_j$ of clusters we delete the edge~$V_iV_j$ (if it exists)
with probability 
\begin{equation}\label{eq:probdelete}
\frac{e_{G'}(V_j,V_i)}{e_{G'}(V_i,V_j)+e_{G'}(V_j,V_i)}.
\end{equation}
Otherwise we delete~$V_jV_i$ (if it exists). We interpret~(\ref{eq:probdelete})
as~0 if $V_iV_j,V_jV_i\notin E(R')$. So if~$R'$ contains at most one of the edges $V_iV_j,V_jV_i$
then we do nothing.
We do this for all unordered pairs of clusters independently and let~$X_i$ be the random variable
which counts the number of outedges of the vertex~$V_i\in R$.
Then
\begin{align}\label{eq:Xi}
\mathbb{E}(X_i)&=\sum_{j\neq i}{\frac{e_{G'}(V_i,V_j)}{e_{G'}(V_i,V_j)+e_{G'}(V_j,V_i)}}
\geq \sum_{j\neq i}{\frac{e_{G'}(V_i,V_j)}{\abs{V_i}\abs{V_j}}}\nonumber\\
& \ge \frac{\abs{R'}}{|G|\abs{V_i}}\sum_{x\in V_i}(d^+_{G'}(x)-\abs{V_0})\\
& \ge (\delta^+(G')/|G|-\epsilon)\abs{R}
\ge(\delta^+(G)/|G|-(2\epsilon+d))\abs{R}\stackrel{(\ref{eq:notsmall})}{\ge} \eps |R|\nonumber.
\end{align}
A Chernoff-type bound (see e.g.~\cite[Cor.~A.14]{ProbMeth}) now implies that there exists
a constant~$\beta=\beta(\eps)$ such that%
     \COMMENT{To see that this Cor can be applied, note that $X_i$ is a sum of independent indicator variables,
one for each $V_j\in N_{R'}(V_i)$.}
\begin{align*}
\mathbb{P}(X_i<(\delta^+(G)/|G|-(3\epsilon+d))\abs{R}) & \le
\mathbb{P}(|X_i-\mathbb{E}(X_i)|>\eps\mathbb{E}(X_i))\\
& \le {\eul}^{-\beta\mathbb{E}(X_i)}\le {\eul}^{-\beta \eps|R|}.
\end{align*}
Writing~$Y_i$ for the random variable which counts the number of inedges of the vertex~$V_i$ in~$R$,
it follows similarly that 
\begin{equation*}
\mathbb{P}(Y_i<(\delta^-(G)/|G|-(3\epsilon+d))\abs{R})\le {\eul}^{-\beta \eps|R|}.
\end{equation*}
Suppose that~$c$ is as in~(d). Consider any pair $V_iV_j\notin U$ of clusters 
such that either $V_iV_j\notin E(R')$ or $V_iV_j,V_jV_i\in E(R')$. (Note that each
$V_iV_j\notin E(R)\cup U$ satisfies one of these properties.)
As before, let~$X_i$ be the random variable which counts the number of outedges of~$V_i$ in~$R$
and let~$Y_j$ be the number of inedges of~$V_j$ in~$R$. Similary as in~(\ref{eq:Xi}) one can show that 
\begin{equation}\label{eq:ore}
\mathbb{E}(X_i+Y_j)\ge \frac{\abs{R'}}{|G|\abs{V_i}}\left( \sum_{x\in V_i}(d^+_{G'}(x)-\abs{V_0})+
\sum_{y\in V_j}(d^-_{G'}(y)-\abs{V_0})\right).
\end{equation}
To estimate this, we will first show that there is a set~$M$ of at least $(1-\eps)|V_i|$ disjoint pairs
$(x,y)$ with $x\in V_i$, $y\in V_j$ and such that $xy\notin E(G)$. Suppose first
that $V_iV_j, V_jV_i\in E(R')$. But then $(V_j,V_i)_G$ is $\eps$-regular
of density at least~$d$ and thus it contains a matching of size at least $(1-\eps)|V_i|$.
As~$G$ is oriented this matching corresponds to a set~$M$ as required.
If $V_iV_j\notin E(R')$ then $(V_i,V_j)_G$ is $\eps$-regular of density less than~$d$ (since $V_iV_j\notin U$).
Thus the complement of $(V_i,V_j)_G$ is $\eps$-regular of density at least~$1-d$ and so contains
a matching of size at least $(1-\eps)|V_i|$ which again corresponds to a set~$M$ as required. 
Together with~(\ref{eq:ore}) this implies that
\begin{align*}
\mathbb{E}(X_i+Y_j)& \ge \frac{\abs{R'}}{|G|\abs{V_i}}\sum_{(x,y)\in M}(d^+_{G'}(x)+d^-_{G'}(y)-2\abs{V_0})\\
& \ge \frac{\abs{R'}}{|G|\abs{V_i}}(c-2(\eps+d)-2\eps)|G|(1-\eps)|V_i|\ge (c-(5\eps+2d))|R|
\stackrel{(\ref{eq:notsmall})}{\ge} \eps |R|.
\end{align*}
Similarly as before a Chernoff-type bound implies that%
     \COMMENT{Apply Chernoff to $X_i$ and $Y_j$ separately. Have to make a further
case distinction for this: if $\mathbb{E}(X_i),\mathbb{E}(Y_j)\ge \eps |R|/2$ we apply Chernoff to both
$X_i$ and $Y_j$. If eg $\mathbb{E}(Y_j)< \eps |R|/2$ we apply it only to~$X_i$.}
\begin{equation*}
\mathbb{P}(X_i+Y_j<(c-(6\eps+2d))\abs{R})\le {\eul}^{-\beta \eps|R|}.
\end{equation*}
As $2|R|^2{\eul}^{-\beta \eps|R|}<1$ if~$M'$ is chosen to be sufficiently large compared to~$\eps$,
this implies that there is some outcome~$R$ which satisfies~(a), (b) and~(d).
But $N_{R'}(V_i)=N_{R}(V_i)$ for every cluster~$V_i$
and so~(\ref{eq:NVi}) implies that $\delta(R)\ge (\delta(G)/|G|-(3\epsilon+2d))|R|$.
Altogether this shows that~$R$ is as required in the lemma.

If neither of the conditions in~(\ref{eq:notsmall}) hold, then~(a), (b) and~(d) are trivial and
one can obtain an oriented graph~$R$ which satisfies~(c) from~$R'$ by arbitrarily
deleting one edge from each double edge.
If for example only the first of the conditions
in~(\ref{eq:notsmall}) holds, then~(b) and~(d) are trivial. To obtain an oriented
graph~$R$ which satisfies~(a) we consider the~$X_i$ as before, but ignore the~$Y_i$
and the sums~$X_i+Y_j$. Again, $N_{R'}(V_i)=N_{R}(V_i)$ for every cluster~$V_i$
and so~(c) is also satisfied. The other cases are similar.
\end{proof}

The oriented graph~$R$ given by Lemma~\ref{lemma:reduced_oriented} is called the \emph{reduced
oriented graph}. The spanning oriented subgraph~$G^*$ of the pure digraph~$G'$ obtained by deleting
all the $V_i$-$V_j$ edges whenever $V_iV_j\in E(R')\setminus E(R)$ is called the \emph{pure oriented
graph}. Given an oriented subgraph $S\subseteq R$, the \emph{oriented subgraph of~$G^*$ corresponding to~$S$}
is the oriented subgraph obtained from~$G^*$ by deleting all those vertices that lie in clusters not
belonging to~$S$ as well as deleting all the $V_i$-$V_j$ edges for all pairs $V_i,V_j$ with
$V_iV_j\notin E(S)$.

In our proof of Theorem~\ref{theorem:main} we will also need the Blow-up lemma. Roughly speaking,
it states the following. Let~$F$ be a graph on~$r$ vertices, let~$K$ be a graph obtained from~$F$
by replacing each vertex of~$F$ with a cluster and replacing each edge with a complete bipartite graph
between the corresponding clusters. Define~$G$ similarly except that the edges of~$F$ now
correspond to dense $\eps$-super-regular pairs. Then every subgraph~$H$ of~$K$ which has bounded maximum degree
is also a subgraph in~$G$. In the original version of
Koml\'os, S\'ark\"ozy and Szemer\'edi~\cite{Komlos_Szemeredi_Blowup_Lemma} $\eps$ has to be sufficiently
small compared to~$1/r$ (and so in particular we cannot take $r=|R|$). We will use a stronger (and more technical) version due to
Csaba~\cite{csaba_06_bollobas_eldridge}, which allows us to take~$r=|R|$ and does not demand
super-regularity. The case when $\Delta=3$ of this is implicit
in~\cite{Csaba_Szemeredi_Bollobas_Eldridge_D3}.
     \COMMENT{The BL is trivial unless $m\ge 1/\eps$ (as otherwise each of the
non-empty bipartite graphs $(V_i,V_j)_{G^*}$ must be complete) and this seems to be
enough for the pf of the BL, ie we don't need the additional assumptions
that $N\gg k$ or $N\gg 1/\eps$.}

In the statement of Lemma~\ref{lemma:blowup_lemma} and later on
we write $0<a_1 \ll a_2 \ll a_3$ to mean that we can choose the constants
$a_1,a_2,a_3$ from right to left. More
precisely, there are increasing functions $f$ and $g$ such that, given
$a_3$, whenever we choose some $a_2 \leq f(a_3)$ and $a_1 \leq g(a_2)$, all
calculations needed in the proof of Lemma~\ref{lemma:blowup_lemma} are valid. 
Hierarchies with more constants are defined in the obvious way.

\begin{lemma}[Blow-up Lemma, Csaba~\cite{csaba_06_bollobas_eldridge}]\label{lemma:blowup_lemma}
For all integers~$\Delta, K_1,K_2,K_3$ and every positive constant~$c$ there exists an integer~$N$
such that whenever $\epsilon,\eps',\delta',d$ are positive constants with
$$0<\epsilon\ll\epsilon'\ll\delta'\ll d\ll 1/\Delta,1/K_1,1/K_2,1/K_3,c 
$$
the following holds. Suppose that~$G^*$ is a graph of order~$n\ge N$ and $V_0,\dots,V_k$ is a partition
of~$V(G^*)$ such that the bipartite graph~$(V_i,V_j)_{G^*}$ is $\epsilon$-regular 
with density either $0$ or~$d$ for all $1\le i<j\le k$. Let~$H$ be a graph on~$n$ vertices with
$\Delta(H)\le \Delta$ and let $L_0\cup L_1\cup\dots\cup L_{k}$ be a partition of~$V(H)$
with $|L_i|=|V_i|=:m$ for every $i=1,\dots, k$. Furthermore, suppose that
there exists a bijection $\phi:L_0\rightarrow V_0$ and a set $I\subseteq V(H)$
of vertices at distance at least~$4$
from each other such that the following conditions hold:
\begin{itemize}
\item[\textup{(C1)}]$\abs{L_0}=\abs{V_0}\leq K_1dn$.
\item[\textup{(C2)}]$L_0\subseteq I$.
\item[\textup{(C3)}]$L_i$ is independent for every $i=1,\dots, k$.
\item[\textup{(C4)}]$\abs{N_H(L_0)\cap L_i}\leq K_2dm$ for every $i=1,\dots, k$.
\item[\textup{(C5)}]For each $i=1,\dots, k$ there exists $D_i\subseteq I\cap L_i$ with
$\abs{D_i}=\delta'm$ and such that for $D:=\bigcup_{i=1}^k D_i$ and all $1\leq i<j\leq k$
\[\abs{\abs{N_H(D)\cap L_i}-\abs{N_H(D)\cap L_j}}<\epsilon m.\]
\item[\textup{(C6)}]If $xy\in E(H)$ and $x\in L_i,y\in L_j$ where $i,j\neq 0$ then $(V_i,V_j)_{G^*}$ is $\epsilon$-regular 
with density~$d$.
\item[\textup{(C7)}]If $xy\in E(H)$ and $x\in L_0, y\in L_j$ then $|N_{G^*}(\phi(x))\cap V_j|\geq cm$.
\item[\textup{(C8)}]For each $i=1,\dots, k$, given any $E_i\subseteq V_i$ with $\abs{E_i}\leq\epsilon' m$
there exists a set $F_i\subseteq(L_i\cap(I\setminus D))$ and a bijection $\phi_i:E_i\rightarrow F_i$
such that $|N_{G^*}(v)\cap V_j|\geq (d-\epsilon)m$ whenever $N_H(\phi_i(v))\cap L_j\neq \emptyset$
(for all $v\in E_i$ and all $j=1,\dots,k$).
\item[\textup{(C9)}]Writing $F:=\bigcup_{i=1}^k F_i$ we have that
$\abs{N_H(F)\cap L_i}\leq K_3\epsilon'm$.
\end{itemize}
Then~$G^*$ contains a copy of~$H$ such that the image of~$L_i$ is~$V_i$ for all $i=1,\dots,k$ and the image of
each $x\in L_0$ is $\phi(x)\in V_0$.
\end{lemma}

The additional properties of the copy of~$H$ in~$G^*$ are not included in the statement of the
lemma in~\cite{csaba_06_bollobas_eldridge} but are stated explicitly in the proof.

Let us briefly motivate the conditions of the Blow-up lemma.
The embedding of~$H$ into~$G$ guaranteed by the Blow-up lemma is found by a randomized
algorithm which first embeds each vertex $x\in L_0$ to~$\phi(x)$
and then successively embeds the remaining vertices of~$H$.
So the image of~$L_0$ will be the exceptional set~$V_0$.
Condition~(C1) requires that there are not too many exceptional vertices and~(C2) ensures
that we can embed the vertices in~$L_0$ without affecting the neighbourhood of other such vertices.
As~$L_i$ will be embedded into~$V_i$ we need to have~(C3). Condition~(C5) gives us a reasonably
large set~$D$ of `buffer vertices' which will be embedded last by the randomized algorithm.
(C6)~requires that edges between vertices of~$H-L_0$ are
embedded into $\epsilon$-regular pairs of density~$d$.
(C7)~ensures that the exceptional vertices have large degree
in all `neighbouring clusters'. 
(C8) and~(C9) allow us to embed those vertices whose set of candidate images
in~$G^*$ has grown very small at some point of the algorithm.
Conditions~(C6), (C8) and~(C9) correspond to a substantial weakening of
the super-regularity that the usual form of the Blow-up lemma requires, namely that whenever~$H$ contains
an edge~$xy$ with and $x\in L_i,y\in L_j$ then $(V_i,V_j)_{G^*}$ is $(\epsilon,d)$-super-regular.

We would like to apply the Blow-up lemma with~$G^*$ being obtained from the underlying graph
of the pure oriented graph by adding the exceptional vertices.
It will turn out that in order to satisfy~(C8), it suffices to ensure that all the edges
of a suitable 1-factor in the reduced oriented graph~$R$ correspond to
$(\eps,d)$-superregular pairs of clusters. A well-known simple fact (see the first part of the proof
of Proposition~\ref{prop:randomsplit}) states that this can be ensured by removing a small proportion
of vertices from each cluster~$V_i$, and so~(C8) will be satisfied.
However, (C6) requires all the edges of~$R$ to correspond to $\eps$-regular pairs of
density precisely~$d$ and not just at least~$d$.
(As remarked by Csaba~\cite{csaba_06_bollobas_eldridge}, it actually suffices that
the densities are close to~$d$ in terms of~$\eps$.)
The second part of the following proposition shows that this too does not pose a problem.

\begin{proposition}\label{prop:randomsplit}
Let $M',n_0,D$ be integers and let $\eps,d$ be positive constants such that
$1/n_0\ll 1/M'\ll \eps\ll d\ll 1/D$. Let~$G$ be an oriented graph of order at
least~$n_0$. Let~$R$ be the reduced oriented graph and let~$G^*$ be the pure oriented
graph obtained by successively applying first the Diregularity lemma with parameters $\epsilon$,
$d$ and~$M'$ to~$G$ and then Lemma~\ref{lemma:reduced_oriented}. Let~$S$ be an oriented subgraph
of~$R$ with $\Delta(S)\le D$. Let~$G'$ be the underlying graph
of~$G^*$. Then one can delete $2D\eps|V_i|$ vertices from each cluster~$V_i$
to obtain subclusters $V'_i\subseteq V_i$ in such a way that~$G'$ contains a subgraph~$G'_S$ whose vertex
set is the union of all the $V'_i$ and such that
\begin{itemize}
\item $(V'_i,V'_j)_{G'_S}$ is $(\sqrt{\eps}, d-4D\eps)$-superregular whenever~$V_iV_j\in E(S)$,
\item $(V'_i,V'_j)_{G'_S}$ is $\sqrt{\eps}$-regular and has density~$d-4D\eps$ 
whenever~$V_iV_j\in E(R)$.
\end{itemize}
\end{proposition}
\proof Consider any cluster~$V_i\in V(S)$ and any neighbour~$V_j$ of~$V_i$ in~$S$.
Recall that $m=|V_i|$.
Let~$d_{ij}$ denote the density of the bipartite subgraph $(V_i,V_j)_{G'}$ of~$G'$ induced
by~$V_i$ and~$V_j$. So $d_{ij}\ge d$ and this bipartite graph
is $\eps$-regular. Thus there are at most $2\eps m$ vertices $v\in V_i$ such that
$||N_{G'}(v)\cap V_j|-d_{ij}m|> \eps m$. So in total there are at most $2D\eps m$
vertices $v\in V_i$ such that $||N_{G'}(v)\cap V_j|-d_{ij}m|> \eps m$ for some neighbour~$V_j$
of~$V_i$ in~$S$. Delete all these vertices as well as some more vertices if necessary to
obtain a subcluster~$V'_i\subseteq V_i$ of size $(1-2D\eps)m=:m'$.
Delete any $2D\eps m$ vertices from each cluster $V_i\in V(R)\setminus V(S)$ to obtain
a subcluster~$V'_i$. It is easy to check that for each edge $V_iV_j\in E(R)$ the graph
$(V'_i,V'_j)_{G'}$ is still $2\eps$-regular and that its density~$d'_{ij}$ satisfies
$$
d':=d-4D\eps< d_{ij}-\eps\le d'_{ij}\le d_{ij}+\eps.
$$
Moreover, whenever $V_iV_j\in E(S)$ and $v\in V'_i$ we have that%
    \COMMENT{The lower bound follows since $\eps m+2D\eps m\le 4D\eps m'$.
To see the upper bound note that $(d_{ij}+\eps)m=(d_{ij}+\eps)m'/(1-2\eps D)
\le (d_{ij}+\eps)m'(1+5\eps D/2)\le (d_{ij}+4\eps D)m'$.}
$$
(d_{ij}-4D\eps)m'\le |N_{G'}(v)\cap V'_j|\le (d_{ij}+4D\eps)m'.
$$
For every pair $V_i,V_j$ of clusters with $V_iV_j\in E(S)$ we now consider a spanning random
subgraph~$G'_{ij}$ of $(V'_i,V'_j)_{G'}$ which is obtained by choosing each edge of
$(V'_i,V'_j)_{G'}$ with probability $d'/d'_{ij}$, independently of the other edges.
Consider any vertex $v\in V'_i$. Then the expected number of
neighbours of~$v$ in~$V'_j$ (in the graph~$G'_{ij}$) is at least%
     \COMMENT{$(d_{ij}-4D\eps)/d'_{ij}\ge (d_{ij}-4D\eps)/(d_{ij}+\eps)\ge
1-5D\eps/(d_{ij}+\eps)\ge 1-5D\eps/d\ge 1-\sqrt{\eps}$.}
$(d_{ij}-4D\eps)d'm'/d'_{ij}\ge (1-\sqrt{\eps})d'm'$.
So we can apply a Chernoff-type bound to see that there exists a constant $c=c(\eps)$ such that
$$\mathbb{P}(|N_{G'_{ij}}(v)\cap V'_j|\le (d'-\sqrt{\eps}) m')\le {\eul}^{-cd'm'}.
$$
Similarly, whenever $X\subseteq V'_i$ and $Y\subseteq V'_j$ are sets of size at least~$2\eps m'$
the expected number of $X$-$Y$ edges in~$G'_{ij}$ is $d_{G'}(X,Y)d'|X||Y|/d'_{ij}$.
Since $(V'_i,V'_j)_{G'}$ is $2\eps$-regular this expected number lies between
$(1-\sqrt{\eps})d'|X||Y|$ and $(1+\sqrt{\eps})d'|X||Y|$.
So again we can use a Chernoff-type bound to see that
$$\mathbb{P}(|e_{G'_{ij}}(X,Y)-d'|X||Y||>\sqrt{\eps}|X||Y|)\le {\eul}^{-c d'|X||Y|}
\le {\eul}^{-4cd'(\eps m')^2}.
$$
Moreover, with probability at least $1/(3m')$ the graph $G'_{ij}$ has its expected
density~$d'$ (see e.g.~\cite[p.~6]{BollobasRG}).
Altogether this shows that with probability at least
$$1/(3m')-2m'{\eul}^{-cd'm'}-2^{2m'}{\eul}^{-4cd'(\eps m')^2}>0
$$
we have that~$G'_{ij}$ is $(\sqrt{\eps},d')$-superregular and has density~$d'$.
Proceed similarly for every pair of clusters forming an edge of~$S$.
An analogous argument applied to a pair~$V_i,V_j$ of clusters with $V_iV_j\in E(R)\setminus E(S)$
shows that with non-zero probability the random subgraph~$G'_{ij}$ is
$\sqrt{\eps}$-regular and has density~$d'$.
Altogether this gives us the desired subgraph~$G'_S$ of~$G'$.
\endproof

\subsection{Overview of the proof of Theorem 3}\label{sec:overview}

Let~$G$ be our given oriented graph. The rough idea of the proof is to apply the Diregularity lemma
and Lemma~\ref{lemma:reduced_oriented} to obtain a reduced oriented graph~$R$ and a pure oriented
graph~$G^*$. The following result of H\"aggkvist implies that~$R$ contains a 1-factor.

\begin{theorem}[H\"aggkvist \cite{HaggkvistHamilton}]\label{theorem:1factor}
Let~$R$ be an oriented graph with $\delta^*(R)> (3|R|-3)/2$.
Then~$R$ is strongly connected and contains a $1$-factor. 
\end{theorem}

So one can apply the Blow-up lemma (together with Proposition~\ref{prop:randomsplit})
to find a 1-factor in $G^*-V_0\subseteq G-V_0$. One now would like to glue the
cycles of this 1-factor together and to incorporate the exceptional vertices to obtain a
Hamilton cycle of~$G^*$ and thus of~$G$. However, we were only able to find a method which incorporates
a set of vertices whose size is small compared to the cluster size~$m$. This is not
necessarily the case for~$V_0$. So we proceed as
follows. We first choose a random partition of the vertex set of~$G$ into two sets~$A$
and~$V(G)\setminus A$ having roughly equal size. We then apply the Diregularity lemma
to~$G-A$ in order to obtain clusters $V_1,\dots,V_k$ and an exceptional set~$V_0$.
We let~$m$ denote the size of these clusters and set $B:=V_1\cup\dots V_k$.
By arguing as indicated above, we can find a Hamilton cycle~$C_B$ in~$G[B]$.
We then apply the Diregularity lemma to $G-B$, but with an~$\eps$ which is small compared to~$1/k$,
to obtain clusters $V'_1,\dots,V'_\ell$ and an exceptional set~$V'_0$.
Since the choice of our partition $A,V(G)\setminus A$ will imply that
$\delta^*(G-B)\ge (3/2+\alpha/2)|G-B|$ we can again argue as before to obtain
a cycle~$C_A$ which covers precisely the vertices in $A':=V'_1\cup\dots\cup V'_\ell$.
Since we have chosen~$\eps$ to be small compared to~$1/k$, the set~$V'_0$ of
exceptional vertices is now small enough to be incorporated into our first cycle~$C_B$.
(Actually, $C_B$ is only determined at this point and not yet earlier on.)
Moreover, by choosing~$C_B$ and~$C_A$ suitably we can ensure that they can be joined
together into the desired Hamilton cycle of~$G$. 


\section{Shifted Walks}\label{sec:shifted_paths}
In this section we will introduce the tools we need in order to glue certain cycles
together and to incorporate the exceptional vertices.
Let~$R^*$ be a digraph%
     \COMMENT{Need digraph here since~$R^*_B$ in the pf of the thm is not necessarily
oriented}
and let~$\mathcal{C}$ be a collection
of disjoint cycles in~$R^*$. We call a closed walk~$W$ in~$R^*$ \emph{balanced \wrt $\mathcal{C}$} if 
\begin{itemize}
\item for each cycle $C\in\mathcal{C}$ the walk~$W$ visits all the vertices on~$C$ an equal
number of times,
\item $W$ visits every vertex of~$R^*$,
\item every vertex not in any cycle from~$\C$ is visited exactly once.
\end{itemize}
Let us now explain why balanced walks are helpful in order to incorporate the exceptional vertices.
Suppose that~$\C$ is a 1-factor of the reduced oriented graph~$R$ and that
$R^*$ is obtained from~$R$ by adding all the exceptional
vertices $v\in V_0$ and adding an edge~$vV_i$ (where~$V_i$ is a cluster) whenever~$v$ sends edges
to a significant proportion of the vertices in~$V_i$, say we add~$vV_i$ whenever~$v$ sends
at least~$cm$ edges to~$V_i$. (Recall that~$m$ denotes the size of the clusters.)
The edges in~$R^*$ of the form~$V_iv$ are
defined in a similar way. Let~$G^c$ be the oriented graph obtained from the pure
oriented graph~$G^*$ by making all the non-empty bipartite subgraphs between the clusters complete
(and orienting all the edges between these clusters in the direction induced by~$R$)
and adding the vertices in~$V_0$ as well as all the
edges of~$G$ between~$V_0$ and $V(G-V_0)$.
Suppose that~$W$ is a balanced closed walk in~$R^*$ which visits all the vertices lying on a cycle~$C\in\C$
precisely $m_C\le m$ times. Furthermore, suppose that $|V_0|\le cm/2$ and that
the vertices in~$V_0$ have distance at least~$3$ from each other on~$W$.
Then by `winding around' each cycle~$C\in \C$ precisely
$m-m_C$ times (at the point when~$W$ first visits~$C$) we can obtain a Hamilton cycle in~$G^c$.
Indeed, the two conditions on~$V_0$ ensure that the neighbours
of each~$v\in V_0$ on the Hamilton cycle can be chosen amongst the at least~$cm$
neighbours of~$v$ in the neighbouring clusters of~$v$ on~$W$ in such a way that they are
distinct for different exceptional vertices.
The idea then is to apply the Blow-up lemma to show that this Hamilton cycle corresponds to
one in~$G$. So our aim is to find such a balanced closed walk
in~$R^*$. However, as indicated in Section~\ref{sec:overview}, the difficulties arising when
trying to ensure that the exceptional vertices lie on this walk will force us to
apply the above argument to the subgraphs induced by a random partition of our given
oriented graph~$G$.

Let us now go back to the case when~$R^*$ is an arbitrary digraph and~$\mathcal{C}$
is a collection of disjoint cycles in~$R^*$. Given vertices~$a,b\in R^*$,
a \emph{shifted $a$-$b$ walk} is a walk of the form
$$
W=aa_1C_1b_1a_2C_2b_2\dots a_tC_tb_tb
$$
where~$C_1,\dots,C_t$ are (not necessarily distinct) cycles from~$\C$ and
$a_i$ is the successor of~$b_i$ on~$C_i$ for all $i\le t$. (We might have $t=0$. So an edge~$ab$
is a shifted $a$-$b$ walk.) We call $C_1,\dots,C_t$ the cycles which are \emph{traversed} by~$W$.
So even if the cycles $C_1,\dots,C_t$ are not distinct, we say that~$W$ traverses~$t$
cycles. Note that for every cycle $C\in\C$
the walk $W-\{a,b\}$ visits the vertices on~$C$ an equal number of times. Thus it will
turn out that by joining the
cycles from~$\C$ suitably via shifted walks and incorporating those vertices of~$R^*$ not covered
by the cycles from~$\C$ we can obtain a balanced closed walk on~$R^*$.

Our next lemma will be used to show that if~$R^*$ is oriented and $\delta^*(R^*)\ge (3/2+\alpha)|R^*|$
then any two vertices of~$R^*$ can be joined by a shifted walk traversing only a
small number of cycles from~$\C$ (see Corollary~\ref{corollary:shiftedPathsCorollary}).
The lemma itself shows that the $\delta^*$ condition implies expansion, and this will give us the 
`expansion with respect to shifted neighbourhoods' we need for the existence of shifted walks. 
The proof of Lemma~\ref{lemma:expansion} is similar to that of Theorem~\ref{theorem:1factor}. 

\begin{lemma}\label{lemma:expansion}
Let~$R^*$ be an oriented graph on $N$ vertices
with $\delta^*(R^*)\geq (3/2+\alpha)N$ for some $\alpha>0$. If $X\subseteq V(R^*)$ is nonempty
and $\abs{X}\leq (1-\alpha)N$ then
$\abs{N^+(X)} \ge \abs{X}+\alpha N/2.$
\end{lemma}

\begin{proof}
For simplicity, we write $\delta:=\delta(R^*)$, $\delta^+:=\delta^+(R^*)$ and $\delta^-:=\delta^-(R^*)$. 
Suppose the assertion is false, \ie there exists $X\subseteq V(R^*)$ with $\abs{X} \leq (1-\alpha)N$ and 
\begin{equation}\label{eq:hc_fails}
\abs{N^+(X)} < \abs{X}+\alpha N/2.
\end{equation}
We consider the following partition of $V(R^*)$:
$$
A:= X\cap N^+(X),\ \ B:=N^+(X)\backslash X,\ \
C:=V(R^*)\backslash (X\cup N^+(X)),\ \ D:=X\backslash N^+(X).
$$ 
\eqref{eq:hc_fails} gives us
\begin{equation}\label{eq:hc_fails_partitioned}
|D|+\alpha N/2 > |B|.
\end{equation}
Suppose $A\neq\emptyset$. Then by an averaging argument there exists $x\in A$ with
$\abs{N^+(x)\cap A}<|A|/2$. Hence $\delta^+ \leq \abs{N^+(x)} < |B| + |A|/2.$
Combining this with~\eqref{eq:hc_fails_partitioned} we get
\begin{equation}\label{eq:expansionBound1}
|A| + |B| + |D| \geq 2\delta^+ - \alpha N/2.
\end{equation}
If $A=\emptyset$ then $N^+(X)=B$ and so~\eqref{eq:hc_fails_partitioned} implies 
$|D| + \alpha N/2 \geq |B| \geq \delta^+$. Thus~(\ref{eq:expansionBound1}) again holds. 
Similarly,%
\COMMENT{By an averaging argument there exists $x\in C$ with $\abs{N^-(x)\cap C}<|C|/2$. Hence
$\delta^- \leq \abs{N^-(x)} < |B| +|C|/2.$
Combining this with \eqref{eq:hc_fails_partitioned} gives
$|B| + |C| +|D| \geq 2\delta^- - \alpha N/2.$}
if $C\neq \emptyset$ then considering the inneighbourhood of a suitable vertex $x \in C$ gives
\begin{equation}\label{eq:expansionBound2}
|B|+|C|+|D| \geq 2\delta^- - \alpha N/2.
\end{equation}
If $C=\emptyset$ then the fact that $\abs{X} \leq (1-\alpha)N$ and~(\ref{eq:hc_fails})
together imply that $D\neq \emptyset$. But then $N^-(D)\subseteq B$ and thus $|B|\ge \delta^-$.
Together with~(\ref{eq:hc_fails_partitioned}) this shows that~(\ref{eq:expansionBound2})
holds in this case too.

If $D=\emptyset$ then trivially $|A|+|B|+|C|=N\geq \delta$. If not, then for any $x\in D$ we 
have $N(x)\cap D=\emptyset$ and hence
\begin{equation}\label{eq:expansionBound3}
|A|+|B|+|C| \geq \abs{N(x)} \geq \delta.
\end{equation}
Combining~\eqref{eq:expansionBound1},~\eqref{eq:expansionBound2} and~\eqref{eq:expansionBound3} gives
\begin{equation*}
3|A|+4|B|+3|C|+2|D| \ge 2\delta^-+2\delta^++2\delta - \alpha N=2\delta^*(R^*)-\alpha N.
\end{equation*}
Finally, substituting~\eqref{eq:hc_fails_partitioned} gives
\[
3N+\alpha N/2 \geq 2\delta^*(R^*)-\alpha N \geq 3N+\alpha N,
\]
which is a contradiction.
\end{proof}

As indicated before, we will now use Lemma~\ref{lemma:expansion}
to prove the existence of shifted walks in~$R^*$ traversing only a small number of
cycles from a given 1-factor of~$R^*$. For this (and later on) the following fact will be useful.
\begin{fact} \label{delta0}
Let $G$ be an oriented graph with $\delta^*(G) \ge (3/2+\alpha)|G|$ for some constant
$\alpha>0$. Then $\delta^0(G)> \alpha |G|$.
\end{fact}
\proof
Suppose that $\delta^- (G) \le \alpha |G|$.
As~$G$ is oriented we have that $\delta^+(G)<|G|/2$ and so $\delta^*(G)<3n/2 + \alpha |G|$,
a contradiction. The proof for $\delta^+(G)$ is similar.
\endproof

\begin{corollary}\label{corollary:shiftedPathsCorollary}
Let~$R^*$ be an oriented graph on $N$ vertices with $\delta^*(R^*)\geq(3/2+\alpha)N$ for some $\alpha>0$
and let~$\mathcal{C}$ be a $1$-factor in~$R^*$. Then for any distinct $x,y\in V(R^*)$
there exists a shifted $x$-$y$ walk traversing at most $2/\alpha$ cycles from~$\mathcal{C}$.
\end{corollary}
\begin{proof}
Let $X_i$ be the set of vertices~$v$ for which there is a shifted 
$x$-$v$ walk which traverses at most~$i$ cycles.
So $X_0=N^+(x) \neq \emptyset$ and $X_{i+1}=N^+(X_i^-) \cup X_i$, where~$X^-_i$
is the set of all predecessors of the vertices in~$X_i$ on the cycles from~$\C$.
Suppose that $|X_i| \le (1-\alpha)N$.
Then Lemma~\ref{lemma:expansion} implies that 
$$
|X_{i+1}| \ge |N^+(X_i^-)| \ge |X_i^-|+\alpha N/2 = |X_i|+\alpha N/2.
$$
So for $i^*:=\lfloor 2/\alpha\rfloor-1$, we must have $|X_{i^*}^-|=|X_{i^*}| \ge (1-\alpha)N$.
But $|N^-(y)| \ge \delta^- (R^*)> \alpha N$ and 
so $N^-(y) \cap X_{i^*}^- \neq \emptyset$.
In other words, $y \in N^+(X_{i^*}^-)$ and so there is a shifted 
$x$-$y$ walk traversing at most $i^*+1$ cycles.
\end{proof}

\begin{corollary}\label{corollary:hamiltonian_cycle}
Let~$R^*$ be an oriented graph with $\delta^*(R^*)\geq(3/2+\alpha)\abs{R^*}$ for some $0<\alpha\le 1/6$
and let~$\C$ be a $1$-factor in~$R^*$. Then~$R^*$ contains a closed walk which is balanced
w.r.t.~$\C$ and meets every vertex at most $|R^*|/\alpha$ times and traverses each edge
lying on a cycle from~$\C$ at least once.
\end{corollary}

\begin{proof}
Let $C_1,\ldots,C_s$ be an arbitrary ordering of the cycles in~$\mathcal{C}$. 
For each cycle~$C_i$ pick a vertex $c_i\in C_i$. Denote by~$c_i^+$ the successor of~$c_i$
on the cycle~$C_i$. Corollary~\ref{corollary:shiftedPathsCorollary} implies that for all~$i$
there exists a shifted $c_i$-$c_{i+1}^+$ walk~$W_i$
traversing at most $2/\alpha$ cycles from~$\mathcal{C}$, where $c_{s+1}:=c_1$. Then the closed walk 
$$W':=c_1^+C_1c_1W_1c_2^+C_2c_2\ldots W_{s-1}c_s^+C_sc_sW_sc_1^+$$
 is balanced w.r.t.~$\C$ by the definition of shifted walks. Since each shifted walk~$W_i$ traverses 
 at most $2/\alpha$ cycles of~$\C$, the closed walk~$W$ meets each vertex at 
 most $(\abs{R^*}/3)(2/\alpha)+1$ times. Let~$W$ denote the walk obtained from~$W'$ by
`winding around' each cycle~$C\in\C$ once more. (That is, for each $C\in\C$
pick a vertex~$v$ on~$C$ and replace one of the occurences of~$v$ on~$W'$ by~$vCv$.) Then~$W$
is still balanced \wrt $\C$, traverses each edge lying on a cycle from~$\C$ at
least once and visits each vertex of~$R^*$ at most
$(\abs{R^*}/3)(2/\alpha)+2\le\abs{R^*}/\alpha$ times as required. 
\end{proof}
%
%
%
\section{Proof of Theorem 3}\label{sec:main_proof}
\subsection{Partitioning~$G$ and applying the Diregularity lemma}
Let~$G$ be an oriented graph on~$n$ vertices with $\delta^*(G)\geq(3/2+\alpha)n$
for some constant $\alpha>0$. Clearly we may assume that $\alpha\ll 1$.
Define positive constants~$\eps,d$ and integers~$M'_A,M'_B$
such that 
\[1/M'_A\ll 1/M'_B\ll\epsilon\ll d\ll \alpha \ll 1.\]
Throughout this section, we will assume that~$n$ is sufficiently large compared to~$M'_A$
for our estimates to hold. Choose a subset~$A\subseteq V(G)$ with
$(1/2-\eps)n\le \abs{A}\le (1/2+\eps)n$ and such
that every vertex $x\in G$ satisfies
$$
\frac{d^+(x)}{n}-\frac{\alpha}{10}\le \frac{\abs{N^+(x)\cap A}}{|A|}
\le \frac{d^+(x)}{n}+\frac{\alpha}{10}
$$
and such that~$N^-(x)\cap A$ satisfies a similar condition.
(The existence of such a set~$A$ can be shown by considering a random partition of~$V(G)$.)
Apply the Diregularity lemma (Lemma~\ref{lemma:diregularity_lemma}) with
parameters~$\epsilon^2$, $d+8\eps^2$ and~$M'_B$ to~$G- A$ to obtain a partition of the vertex set
of~$G-A$ into $k\ge M'_B$ clusters $V_1,\dots,V_k$ and an exceptional set~$V_0$.
Set $B:=V_1\cup\ldots\cup V_k$ and $m_B:=\abs{V_1}=\dots=\abs{V_k}$. 
Let~$R_B$ denote the reduced oriented graph obtained by an application of
Lemma~\ref{lemma:reduced_oriented} and let~$G^*_B$ be the pure oriented graph. Since
$\delta^+(G-A)/|G-A|\ge \delta^+(G)/n-\alpha/9$ by our choice of~$A$,
Lemma~\ref{lemma:reduced_oriented} implies that
\begin{equation}\label{eq:delta+RB}
\delta^+(R_B)\geq (\delta^+(G)/n-\alpha/8)|R_B|.
\end{equation}
Similarly
\begin{equation}\label{eq:delta-RB}
\delta^-(R_B)\geq (\delta^-(G)/n-\alpha/8)|R_B|
\end{equation}
and $\delta(R_B)\geq (\delta(G)/n-\alpha/4)|R_B|$. Altogether this implies that
\begin{equation}\label{eq:delta*RB}
\delta^*(R_B)\geq(3/2+\alpha/2)|R_B|.
\end{equation}
So Theorem~\ref{theorem:1factor} gives us a 1-factor~$\mathcal{C}_B$ of~$R_B$.
We now apply Proposition~\ref{prop:randomsplit} with~$\C_B$ playing the role of~$S$,
$\eps^2$ playing the role of~$\eps$ and~$d+8\eps^2$ playing the role of~$d$.
This shows that by adding at most~$4\eps^2 n$ further vertices to the exceptional set~$V_0$ we may assume
that each edge of~$R_B$ corresponds to an $\eps$-regular pair of density~$d$
(in the underlying graph of~$G^*_B$) and that each edge
in the union~$\bigcup_{C\in \C_B} C\subseteq R_B$ of all the cycles from~$\C_B$
corresponds to an $(\eps, d)$-superregular pair.
(More formally, this means that we replace the clusters with the subclusters
given by Proposition~\ref{prop:randomsplit} and replace~$G^*_B$ with its
oriented subgraph obtained by deleting all edges not corresponding to edges of the
graph $G'_{\C_B}$ given by Proposition~\ref{prop:randomsplit}, i.e.~the
underlying graph of~$G^*_B$ will now be~$G'_{\C_B}$.)
Note that the new exceptional set now satisfies $|V_0|\le \eps n$.

Apply Corollary~\ref{corollary:hamiltonian_cycle} with $R^*:=R_B$ to find a
closed walk~$W_B$ in~$R_B$ which is balanced w.r.t.~$\C_B$, meets every cluster
at most $2|R_B|/\alpha$ times%
     \COMMENT{It is $2|R_B|/\alpha$ instead of $|R_B|/\alpha$ since we apply
Corollary~\ref{corollary:hamiltonian_cycle} with $\alpha/2$ playing the role of~$\alpha$.} 
and traverses all the edges lying on a cycle from~$\C_B$ at least once.

Let $G^c_B$ be the oriented graph obtained from~$G^*_B$ by adding all the $V_i$-$V_j$ edges
for all those pairs $V_i,V_j$ of clusters with $V_iV_j\in E(R_B)$.
Since $2|R_B|/\alpha\ll m_B$, we could make~$W_B$ into a
Hamilton cycle of~$G^c_B$ by `winding around' each cycle from~$\C_B$
a suitable number of times.  We could then apply the Blow-up lemma to show that this
Hamilton cycle corresponds to one in $G^*_B$. However, as indicated in Section~\ref{sec:overview},
we will argue slightly differently as it is not clear how to incorporate all the
exceptional vertices by the above approach.

Set $\eps_A:=\epsilon /|R_B|$. 
Apply the Diregularity lemma with parameters $\eps_A^2$, $d+8\eps_A^2$
and~$M'_A$ to~$G[A\cup V_0]$ to obtain a partition of the vertex set of~$G[A\cup V_0]$
into $\ell\ge M'_A$
clusters $V'_1,\dots,V'_\ell$ and an exceptional set~$V'_0$. Let $A':=V'_1\cup\dots\cup V'_\ell$,
let~$R_{A}$ denote the reduced oriented graph obtained from Lemma~\ref{lemma:reduced_oriented}
and let~$G^*_{A}$ be the pure oriented graph.
Similarly as in~(\ref{eq:delta*RB}), Lemma~\ref{lemma:reduced_oriented}
implies that $\delta^*(R_A)\geq(3/2+\alpha/2)|R_A|$ 
and so, as before, we can apply Theorem~\ref{theorem:1factor} to find a
1-factor~$\mathcal{C}_A$ of~$R_A$. Then as before,
Proposition~\ref{prop:randomsplit} implies that by adding at most~$4\eps_A^2 n$
further vertices to the exceptional set~$V'_0$ we may assume
that each edge of~$R_A$ corresponds to an $\eps_A$-regular pair of density~$d$ and that each edge
in the union~$\bigcup_{C\in \C_A} C\subseteq R_A$ of all the cycles from~$\C_A$
corresponds to an $(\eps_A, d)$-superregular pair. So we now have that
\begin{equation}\label{eq:V'0}
|V'_0|\le \eps_A n=\eps n/|R_B|.
\end{equation}
Similarly as before, Corollary~\ref{corollary:hamiltonian_cycle}
gives us a closed walk~$W_A$ in~$R_A$ which is balanced w.r.t.~$\C_A$, meets every cluster
at most $2|R_A|/\alpha$ times and traverses all the edges lying on a cycle
from~$\C_A$ at least once.

\subsection{Incorporating~$V'_0$ into the walk~$W_B$}\label{subsec:exceptional_set}
Recall that the balanced closed walk~$W_B$ in~$R_B$ corresponds to a Hamilton cycle in~$G^c_B$.
Our next aim is to extend this walk to one which corresponds to a Hamilton cycle which
also contains the vertices in~$V'_0$. (The Blow-up lemma will imply that the latter Hamilton
cycle corresponds to one in~$G[B\cup V'_0]$.) We do this by extending~$W_B$ into a walk
on a suitably defined digraph~$R^*_B\supseteq R_B$ with vertex set $V(R_B)\cup V'_0$ in such
a way that the new walk is balanced
\wrt $\mathcal{C}_B$. $R^*_B$ is obtained from the union of~$R_B$ and the set~$V'_0$ by adding
an edge~$vV_i$ between a vertex $v\in V'_0$ and a cluster $V_i\in V(R_B)$ whenever
$\abs{N^+_G(v)\cap V_i}>\alpha m_B/10$ and adding the edge~$V_iv$ whenever
$\abs{N^-_G(v)\cap V_i}>\alpha m_B/10$. Thus\COMMENT{All we actually need is $|N^+_{R^*_B}(v)|>1$, $|N^-_{R^*_B}(v)|>1$.}
\[
|N^+_G(v)\cap B|\le |N^+_{R^*_B}(v)|m_B+|R_B|\alpha m_B/10.
\]
Hence
\begin{align}\label{eq:except_vertex_degree}
|N^+_{R^*_B}(v)| & \ge |N^+_G(v)\cap B|/m_B-\alpha |R_B|/10
\ge |N^+_G(v)\cap B||R_B|/|B|-\alpha |R_B|/10\notag \\
& \ge (|N^+_{G-A}(v)|-|V_0|)|R_B|/|G-A|-\alpha |R_B|/10\notag \\
& \ge (\delta^+(G)/n-\alpha/2)|R_B| \ge \alpha|R_B|/2.
\end{align}
(The penultimate inequality follows from the choice of~$A$ and the final one from Fact~\ref{delta0}.) 
Similarly 
\[ 
|N^-_{R^*_B}(v)| \ge \alpha |R_B|/2.
\]
Given a vertex $v\in V'_0$ pick $U_1\in N^+_{R^*_B}(v)$, $U_2\in N^-_{R^*_B}(v)\backslash\{U_1\}$. 
Let $C_1$ and $C_2$ denote the cycles from $\C_B$ containing $U_1$ and $U_2$ respectively. 
Let $U_1^-$ be the predecessor of $U_1$ on $C_1$, and $U_2^+$ be the successor of $U_2$ on $C_2$. 
\eqref{eq:except_vertex_degree} implies that we can ensure $U_1^- \neq U_2^+$.
(However, we may have $C_1=C_2$.) 
Corollary~\ref{corollary:shiftedPathsCorollary} gives us a shifted walk $W_v$ from $U_1^-$ to $U_2^+$ 
traversing at most $4/\alpha$ cycles%
    \COMMENT{Get $4/\alpha$ since we apply the cor with $\alpha/2$.}
of~$\C_B$. 
\begin{figure}\label{fig:exceptvs}
\centering\footnotesize
\psfrag{1}[][]{$C_1$}
\psfrag{2}[][]{$C_2$}
\psfrag{3}[][]{$U_1^-$}
\psfrag{4}[][]{$U_1$}
\psfrag{5}[][]{$U_2$}
\psfrag{6}[][]{$U_2^+$}
\psfrag{7}[][]{$v$}
\psfrag{8}[][]{$W_v$}
\includegraphics[scale=0.7]{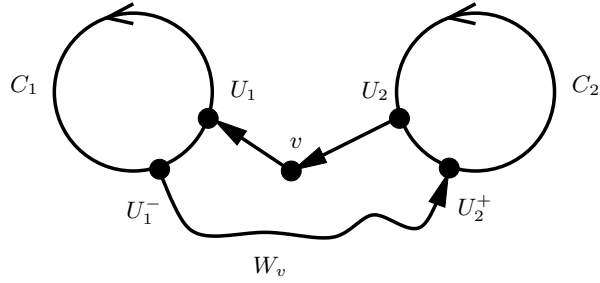}
\caption{Incorporating the exceptional vertex~$v$.}
\end{figure}
To incorporate~$v$ into the walk~$W_B$, recall that~$W_B$ traverses all those edges of~$R_B$
which lie on cycles from~$\C_B$ at least once. Replace one of the occurences
of~$U^-_1U_1$ on~$W_B$ with the walk
$$
W'_v:=U^-_1W_vU^+_2C_2U_2vU_1C_1U_1,
$$
i.e.~the walk that goes from~$U^-_1$ to $U^+_2$ along the shifted walk~$W_v$, it then winds
once around~$C_2$ but stops in~$U_2$,
then it goes to~$v$ and further to~$U_1$, and finally it winds around~$C_1$.
The walk obtained from~$W_B$ by including~$v$ in this way is still balanced w.r.t.~$\C_B$,
i.e.~each vertex in~$R_B$ is visited the same number of times as every other vertex
lying on the same cycle
from~$\mathcal{C}_B$. 
We add the extra loop around~$C_1$ because when applying the Blow-up lemma we will
need the vertices in~$V'_0$ to be at a distance of at least~4 from each other.
Using this loop, this can be ensured as follows. 
After we have incorporated~$v$ into~$W_B$ we `ban' all the~6 edges of (the new walk)~$W_B$
whose endvertices both have distance at most~3 from~$v$. The extra loop ensures that every edge in
each cycle from~$\mathcal{C}$ has at least one occurence in~$W_B$ which is not banned.
(Note that we do not have to add an extra loop around~$C_2$ since if $C_2\neq C_1$ then all
the banned edges of~$C_2$ lie on~$W'_v$ but each edge of~$C_2$ also occurs on the
original walk~$W_B$.)
Thus when incorporating the next exceptional vertex we can always pick an occurence of an edge which
is not banned to be replaced by a longer walk. (When incorporating~$v$ we picked~$U^-_1U_1$.)
Repeating this argument, we can incorporate all the exceptional vertices in~$V'_0$
into~$W_B$ in such a way that all the vertices of~$V'_0$ have distance at least~4 on
the new walk~$W_B$.

Recall that~$G^c_B$ denotes the oriented graph obtained from the pure oriented graph~$G^*_B$
by adding all the $V_i$-$V_j$ edges for all those pairs $V_i,V_j$ of clusters with $V_iV_j\in E(R_B)$.
Let~$G^c_{B\cup V'_0}$ denote the graph
obtained from~$G^c_B$ by adding all the $V'_0$-$B$ edges of~$G$ as well as all
the $B$-$V'_0$ edges of~$G$. Moreover, recall that the vertices in~$V'_0$ have distance
at least~$4$ from each other on~$W_B$ and
$|V'_0|\le \eps n/|R_B|\ll \alpha m_B/20$ by~(\ref{eq:V'0}).%
     \COMMENT{Need $\alpha m_B/20$ instead of $\alpha m_B/2$ since we apply the argument of
Section~\ref{sec:shifted_paths} with $c=\alpha/10$.}
As already observed at the beginning of Section~\ref{sec:shifted_paths}, altogether this shows
that by winding around each cycle from~$\C_B$, one can obtain
a Hamilton cycle~$C^c_{B\cup V'_0}$ of~$G^c_{B\cup V'_0}$ from the walk~$W_B$, provided that~$W_B$ visits
any cluster $V_i\in R_B$ at most~$m_B$ times. To see that the latter condition holds, recall
that before we incorporated the exceptional
vertices in~$V'_0$ into~$W_B$, each cluster was visited at most $2|R_B|/\alpha$
times. When incorporating an exceptional vertex we replaced an edge of~$W_B$
by a walk whose interior visits every cluster at most $4/\alpha+2\le 5/\alpha$ times.%
    \COMMENT{Its $+2$ and not $+1$ in the case when the cycles $C_1$ and $C_2$ are the same.}
Thus the final walk~$W_B$ visits each cluster $V_i\in R_B$ at most
\begin{equation}\label{eq:visit}
2|R_B|/\alpha + 5|V'_0|/\alpha \stackrel{(\ref{eq:V'0})}{\le}
6\eps n/(\alpha |R_B|) \le \sqrt{\eps} m_B
\end{equation}
times. Hence we have the desired Hamilton cycle~$C^c_{B\cup V'_0}$ of~$G^c_{B\cup V'_0}$.
Note that~(\ref{eq:visit}) implies that we can choose~$C^c_{B\cup V'_0}$ in such a way
that for each cycle $C\in \C_B$ there is subpath~$P_C$ of~$C^c_{B\cup V'_0}$
which winds around~$C$ at least
\begin{equation}\label{eq:PC}
(1- \sqrt{\eps}) m_B
\end{equation}
times in succession.

\subsection{Applying the Blow-up lemma to find a Hamilton cycle in~$G[B\cup V'_0]$}
Our next aim is to use the Blow-up lemma to show that~$C^c_{B\cup V'_0}$ corresponds to a
Hamilton cycle in $G[B\cup V'_0]$. Recall that $k=|R_B|$ and that for each exceptional vertex~$v\in V'_0$
the outneighbour~$U_1$ of~$v$ on~$W_B$ is distinct from its inneighbour~$U_2$
on~$W_B$.
We will apply the Blow-up lemma with~$H$
being the underlying graph of~$C^c_{B\cup V'_0}$ and~$G^*$ being
the graph obtained from the underlying graph of~$G^*_B$ by adding all the vertices~$v\in V'_0$
and joining each such~$v$ to all the vertices in $N^+_G(v)\cap U_1$ as well as to all
the vertices in $N^-_G(v)\cap U_2$. Recall that after applying the Diregularity lemma to
obtain the clusters $V_1,\dots,V_k$ we used Proposition~\ref{prop:randomsplit}
to ensure that each edge of~$R_B$ corresponds to an $\eps$-regular pair of density~$d$
(in the underlying graph of~$G^*_B$ and thus also in~$G^*$) and that each edge of the
union $\bigcup_{C\in \C_B} C\subseteq R_B$ of all the cycles from~$\C_B$ corresponds to
an $(\eps, d)$-superregular pair. 

$V'_0$ will play the role of~$V_0$ in the Blow-up lemma and we take~$L_0,L_1,\dots,L_k$ to
be the partition of~$H$ induced by~$V'_0,V_1,\dots,V_k$. $\phi: L_0\to V'_0$ will be the
obvious bijection (i.e.~the identity). To define the set $I\subseteq V(H)$ of vertices
of distance at least~$4$ from each other which is used in the Blow-up lemma, let~$P'_C$
be the subpath of~$H$ corresponding to~$P_C$ (for all $C\in \C_B$).
For each $i=1,\dots,k$, let $C_i\in \C_B$ denote the cycle containing~$V_i$ and let
$J_i\subseteq L_i$ consist of all those vertices in~$L_i\cap V(P'_{C_i})$ which have distance at least~4
from the endvertices of~$P'_{C_i}$. Thus in the graph~$H$
each vertex $u\in J_i$ has one of its neighbours in the set~$L^-_i$ corresponding to the predecessor
of~$V_i$ on~$C_i$ and its other neighbour in the set~$L^+_i$ corresponding to the successor of~$V_i$ on~$C_i$.
Moreover, all the vertices in~$J_i$ have distance at least~4 from all the vertices in~$L_0$
and (\ref{eq:PC}) implies that $\abs{J_i} \ge 9m_B/10$.
It is easy to see that one can greedily choose a set $I_i\subseteq J_i$ of size $m_B/10$ such that the vertices in
$\bigcup_{i=1}^k I_i$ have distance at least~4 from each other.%
     \COMMENT{Indeed, suppose that we are about to choose~$I_i$. For each of the at most 6 clusters~$V_j$
which lie on~$C_i$ and have distance at most~3 from~$V_i$ on~$C_i$ the set~$I_j$ forbids at most
$|I_j|$ vertices in~$J_i$ (if~$I_j$ was chosen before~$I_i$). This leaves us with
$9m_B/10-6m_B/10=3m_B/10$ vertices from~$J_i$. As the vertices in~$J_i$ have distance at least~3
from each other, we can choose every second of the remaining vs in~$J_i$, ie get a set
of at least $3m_B/20>m_B/10$ vertices.} 
We take $I:=L_0\cup \bigcup_{i=1}^k I_i$.

Let us now check conditions~(C1)--(C9). (C1) holds with~$K_1:=1$ since
$|L_0|=|V'_0|\le \eps_A n= \eps n/k\le d|H|$. (C2) holds by definition of~$I$.
(C3) holds since~$H$ is a Hamilton cycle in~$G^c_{B\cup V'_0}$ 
(c.f.~the definition of the graph~$G^c_{B\cup V'_0}$). This also implies that for every edge $xy\in H$
with $x\in L_i,y\in L_j$ ($i,j\ge 1$) we must have that $V_iV_j\in E(R_B)$. Thus~(C6) holds
as every edge of~$R_B$ corresponds to an $\eps$-regular pair of clusters having density~$d$.
(C4) holds with $K_2:=1$ because 
\[\abs{N_H(L_0)\cap L_i}\leq 2\abs{L_0}= 2\abs{V'_0}\stackrel{(\ref{eq:V'0})}{\le}
2\eps n/|R_B| \le 5\epsilon m_B\le dm_B.
\]
For~(C5) we need to find a set $D\subseteq I$ of buffer vertices. Pick
any set $D_i\subseteq I_i$ with $\abs{D_i}=\delta' m_B$ and let $D:=\bigcup_{i=1}^k D_i$.
Since $I_i\subseteq J_i$ we have that $\abs{N_H(D)\cap L_j}=2\delta'm_B$
for all $j=1,\dots,k$. Hence
$$\abs{\abs{N_H(D)\cap L_i}-\abs{N_H(D)\cap L_j}}=0
$$
for all $1\leq i<j\leq k$ and so~(C5) holds. (C7) holds with $c:=\alpha/10$ by our
choice~$U_1\in N^+_{R^*_B}(v)$ and~$U_2\in N^-_{R^*_B}(v)$ of the neighbours of each
vertex $v\in V'_0$ in the walk~$W_B$ (c.f.~the definition of the graph~$R^*_B$). 

(C8) and~(C9) are now the only conditions we need to check. Given a set $E_i\subseteq V_i$
of size at most~$\epsilon' m_B$, we wish to find
$F_i\subseteq (L_i\cap(I\setminus D))=I_i\setminus D$ and a
bijection $\phi_i:E_i\rightarrow F_i$ such that every $v\in E_i$ has a large number
of neighbours in every cluster~$V_j$ for which~$L_j$ contains a neighbour of~$\phi_i(v)$. Pick any set
$F_i\subseteq I_i\setminus D$ of size~$|E_i|$. (This can be done since
$\abs{D\cap I_i}=\delta' m_B$ and so $\abs{I_i\setminus D}\ge
m_B/10-\delta' m_B\gg \eps' m_B$.) Let $\phi_i:E_i\rightarrow F_i$ be an arbitrary bijection.
To see that~(C8) holds with these choices, consider any vertex $v\in E_i\subseteq V_i$
and let~$j$ be such that~$L_j$ contains a neighbour of~$\phi_i(v)$ in~$H$. Since
$\phi_i(v)\in F_i\subseteq I_i\subseteq J_i$, this means that~$V_j$ must be a neighbour
of~$V_i$ on the cycle $C_i\in \C_B$ containing~$V_i$. But this implies that
$|N_{G^*}(v)\cap V_j|\ge (d-\eps)m_B$ since each edge of the union~$\bigcup_{C\in \C_B} C\subseteq R_B$
of all the cycles from~$\C_B$ corresponds to an $(\eps, d)$-superregular pair in~$G^*$. 

Finally, writing $F:= \bigcup_{i=1}^k F_i$ we have
\[\abs{N_H(F)\cap L_i}\le 2\epsilon'm_B\]
(since $F_j\subseteq J_j$ for each $j=1,\dots,k$) and so~(C9) is satisfied with $K_3:=2$.
Hence (C1)--(C9) hold and so we can apply the Blow-up lemma to obtain a Hamilton cycle
in~$G^*$ such that the image of~$L_i$ is~$V_i$ for all $i=1,\dots,k$
and the image of each $x\in L_0$ is $\phi(x)\in V_0$. 
(Recall that~$G^*$ was obtained from the underlying graph of~$G^*_B$ by adding
all the vertices~$v\in V'_0$ and joining each such~$v$ to all the vertices in
$N^+_G(v)\cap U_1$ as well as to all the vertices in $N^-_G(v)\cap U_2$,
where~$U_1$ and~$U_2$ are the neighbours of~$v$ on the walk~$W_B$.)
Using the fact that~$H$ was obtained from the (directed) Hamilton cycle~$C^c_{B\cup V'_0}$
and since $U_1\neq U_2$ for each $v\in V'_0$, it is easy
to see that our Hamilton cycle in~$G^*$ corresponds to a (directed)
Hamilton cycle~$C_B$ in~$G[B\cup V'_0]$.

\subsection{Finding a Hamilton cycle in~$G$}
The last step of the proof is to find a Hamilton cycle in~$G[A']$ which can be connected with~$C_B$
into a Hamilton cycle of~$G$. Pick an arbitrary edge~$v_1v_2$ on~$C_B$ and add an extra vertex~$v^*$
to~$G[A']$ with outneighbourhood $N_G^+(v_1)\cap A'$ and inneighbourhood $N_G^-(v_2)\cap A'$. 
A Hamilton cycle~$C_A$
in the digraph thus obtained from~$G[A']$ can be extended to a Hamilton cycle of~$G$ by replacing~$v^*$
with $v_2C_Bv_1$. To find such a Hamilton cycle~$C_A$, we can argue as before.
This time, there is only one exceptional vertex, namely~$v^*$, which we incorporate into
the walk~$W_A$. Note that by our choice of $A$ and $B$ the analogue of~(\ref{eq:except_vertex_degree})
is satisfied and so this can be done as before.%
    \COMMENT{Check!}
We then use the Blow-up lemma to obtain the desired Hamilton cycle~$C_A$
corresponding to this walk.

\section{Proof of Theorem 4}\label{sec:Ore}
The following observation guarantees that every oriented graph as in Theorem~\ref{thm:Ore}
has large minimum semidegree.
\begin{fact}\label{fact:ore}
Suppose that $0<\alpha<1$ and that~$G$ is an oriented graph such
that $d^+(x)+d^-(y)\ge (3/4+\alpha)|G|$ whenever $xy\notin E(G)$. Then
$\delta^0(G)\ge |G|/8+\alpha |G|/2$.
\end{fact}
\proof
Suppose not. We may assume that $\delta^+(G)\le \delta^-(G)$. Pick a vertex $x$ with
$d^+(x)=\delta^+(G)$. Let~$Y$ be the set of all those vertices~$y$ with $xy\notin E(G)$.
Thus $|Y|\ge 7|G|/8-\alpha |G|/2$. Moreover, $d^-(y)\ge (3/4+\alpha)|G|-d^+(x)\ge 5|G|/8+\alpha |G|/2$.
Hence%
     \COMMENT{get $35|G|^2/64+7\alpha |G|^2/16-5\alpha |G|^2/16-\alpha^2 |G|^2/4>35|G|^2/64$
since $\alpha< 1/4$ as otherwise the Ore-type condition doesn't make sense}
$e(G)\ge |Y|(5|G|/8+\alpha |G|/2)> 35|G|^2/64$,
a contradiction.
\endproof

The proof of Theorem~~\ref{thm:Ore} is similar to that of Theorem~\ref{theorem:main}.
Fact~\ref{fact:ore} and Lemma~\ref{lemma:reduced_oriented} together imply that
the reduced oriented graph $R_A$ (and similarly~$R_B$) has minimum semidegree at least $|R|/8$
and it inherits the Ore-type condition from~$G$ (i.e.~it satisfies condition~(d)
of Lemma~\ref{lemma:reduced_oriented} with $c=3/4+\alpha$). Together with Lemma~\ref{lemma:ore} below (which is an
analogue of Lemma~\ref{lemma:expansion}) this implies that $R_A$ (and $R_B$ as well) is an expander
in the sense that $\abs{N^+(X)} \ge \abs{X}+\alpha |R_A|/2$ for all $X\subseteq V(R_A)$ with%
     \COMMENT{small sets expand by the minimum semidegree condition}
$\abs{X}\leq (1-\alpha)|R_A|$. 
In particular, $R_A$ (and similarly $R_B$) has a $1$-factor:
To see this, note that the above expansion property together with Fact~\ref{fact:ore}
imply that for any $X \subseteq V(R_A)$, we have $|N^+_{R_A}(X)| \ge |X|$.
Together with Hall's theorem, this means that the following bipartite graph $H$
has a perfect matching: the vertex classes $W_1,W_2$ are $2$ copies of $V(R_A)$
and we have an edge in~$H$ between $w_1 \in W_1$ and $w_2 \in W_2$ if there is an edge from 
$w_1$ to $w_2$ in $R_A$. But clearly a perfect matching in $H$ corresponds to a $1$-factor in~$R_A$.
Using these facts, one can now argue precisely as in the proof of
Theorem~\ref{theorem:main}.

\begin{lemma}\label{lemma:ore}
Suppose that $0< \eps\ll\alpha\ll 1$. Let~$R^*$ be an oriented graph on $N$ vertices
and let~$U$ be a set of at most $\eps N^2$ ordered pairs of vertices of~$R^*$.
Suppose that $d^+(x)+d^-(y)\ge (3/4+\alpha)N$ for all $xy\notin E(R^*)\cup U$.
Then any $X\subseteq V(R^*)$ with $\alpha N\le \abs{X}\leq (1-\alpha)N$ satisfies
$\abs{N^+(X)} \ge \abs{X}+\alpha N/2$.
\end{lemma}
\proof
The proof is similar to that of Lemma~\ref{lemma:expansion}.
Suppose that Lemma~\ref{lemma:ore} does not hold and let $X\subseteq V(R^*)$ with
$\alpha N\le \abs{X}\leq (1-\alpha)N$ be such that
\begin{equation}\label{eq:sizeX}
\abs{N^+(X)} < \abs{X}+\alpha N/2.
\end{equation}
Call a vertex of~$R^*$ \emph{good} if it lies in at most $\sqrt{\eps}N$ pairs from~$U$.
Thus all but at most $2\sqrt{\eps}N$ vertices of~$R^*$ are good.
As in the proof of Lemma~\ref{lemma:expansion} we consider the following partition 
of $V(R^*)$:
$$
A:= X\cap N^+(X),\ \ B:=N^+(X)\backslash X,\ \
C:=V(R^*)\backslash (X\cup N^+(X)),\ \ D:=X\backslash N^+(X).
$$ 
(\ref{eq:sizeX}) implies
\begin{equation}\label{eq:DB}
|D|+\alpha N/2 > |B|.
\end{equation}
Suppose first that $|D|>2\sqrt{\eps}N$. It is easy to see that there are vertices $x\neq y$ in~$D$
such that%
     \COMMENT{exists since otherwise $|U|\ge \binom{|D|}{2}> \eps N^2$}
$xy,yx\notin U$. Since no edge of~$R^*$ lies within~$D$ we have $xy,yx\notin E(R^*)$ and so
$d(x)+d(y)\ge 3N/2+2\alpha N$. In particular, at least one of $x,y$ has degree at least
$3N/4+\alpha N$. But then
\begin{equation}\label{eq:ABC}
|A|+|B|+|C| \ge 3N/4+\alpha N.
\end{equation}
If $|D|\le 2\sqrt{\eps}N$ then $|A|+|B|+|C|\ge N-|D|$ and so~(\ref{eq:ABC}) still holds
with room to spare. Note that~(\ref{eq:DB}) and~(\ref{eq:ABC}) together imply that
$2|A|+2|C|\ge 3N/2+2\alpha N-2|B|\ge 3N/2-|B|-|D|\ge N/2$. Thus at least one
of $A,C$ must have size at least~$N/8$. In particular, this implies that one of the
following~3 cases holds.

\medskip

\noindent
\textbf{Case~1.} $|A|,|C|> 2\sqrt{\eps}N$.

\smallskip

\noindent
Let~$A'$ be the set of all good vertices in~$A$. By an averaging argument there exists $x\in A'$ with
$\abs{N^+(x)\cap A'}<|A'|/2$. Since $N^+(A)\subseteq A\cup B$ this implies
that $\abs{N^+(x)} < |B| +|A\setminus A'|+ |A'|/2$.
Let $C'\subseteq C$ be the set of all those vertices $y\in C$ with $xy\notin U$.
Thus $|C\setminus C'|\le \sqrt{\eps}N$ since~$x$ is good. By an averaging argument there
exists $y\in C'$ with $\abs{N^-(y)\cap C'}<|C'|/2$. But $N^-(C)\subseteq B\cup C$ and so
$\abs{N^-(y)} < |B| +|C\setminus C'|+|C'|/2$. Moreover, $d^+(x)+d^-(y)\ge 3N/4+\alpha N$
since $xy\notin E(R^*)\cup U$. Altogether this shows that
$$|A'|/2+|C'|/2+2|B|\ge d^+(x)+d^-(y) -|A\setminus A'|-|C\setminus C'|\ge 3N/4+\alpha N/2.
$$
Together with~(\ref{eq:ABC}) this implies that $3|A|+6|B|+3|C|\ge 3N +3\alpha N$,
which in turn together with~(\ref{eq:DB}) yields $3|A|+3|B|+3|C|+3|D|\ge 3N +3\alpha N/2$,
a contradiction.

\medskip

\noindent
\textbf{Case~2.} $|A|> 2\sqrt{\eps}N$ and $|C|\le 2\sqrt{\eps}N$.

\smallskip

\noindent
As in Case~1 we let~$A'$ be the set of all good vertices in~$A$ and pick $x\in A'$ with
$\abs{N^+(x)} < |B| +|A\setminus A'|+ |A'|/2$. Note that~(\ref{eq:DB}) 
implies that $|D|>N-|X|-|C|-\alpha N/2\ge \sqrt{\eps}N$. Pick any $y\in D$ such that $xy\notin U$.
Then $xy\notin E(R^*)$ since~$R^*$ contains no
edges from~$A$ to~$D$. Thus $d^+(x)+d^-(y)\ge 3N/4+\alpha N$. Moreover,
$N^-(y)\subseteq B\cup C$. Altogether this gives
$$|A'|/2+2|B|\ge d^+(x)+d^-(y)-|A\setminus A'|-|C|\ge 3N/4+\alpha N/2.
$$
As in Case~1 one can combine this with~~(\ref{eq:ABC}) and~(\ref{eq:DB}) to get a
contradiction.

\medskip

\noindent
\textbf{Case~3.} $|A|\le 2\sqrt{\eps}N$ and $|C|> 2\sqrt{\eps}N$.

\smallskip

\noindent
This time we let~$C'$ be the set of all good vertices in~$C$ and pick
$y\in C'$ with $\abs{N^-(y)\cap C'}<|C'|/2$. Hence $\abs{N^-(y)} < |B| +|C\setminus C'|+ |C'|/2$.
Moreover, we must have $|D|= |X|-|A|>\sqrt{\eps}N$. Pick any $x\in D$ such that $xy\notin U$.
Then $xy\notin E(R^*)$ since~$R^*$ contains no
edges from~$D$ to~$C$. Thus $d^+(x)+d^-(y)\ge 3N/4+\alpha N$. Moreover,
$N^+(x)\subseteq A\cup B$. Altogether this gives
$$|C'|/2+2|B|\ge d^+(x)+d^-(y)-|A|-|C\setminus C'|\ge 3N/4+\alpha N/2,
$$
which in turn yields a contradiction as before.
\endproof

\section{Acknowledgement}
We are grateful to Peter Keevash for pointing out an argument for the `shifted expansion property'
which is simpler than the one presented in an earlier version of this manuscript.

%
%
%

\medskip

{\footnotesize \obeylines \parindent=0pt

Luke Kelly, Daniela K\"{u}hn \& Deryk Osthus
School of Mathematics
University of Birmingham
Edgbaston
Birmingham
B15 2TT
UK
}

{\footnotesize \parindent=0pt

\it{E-mail addresses}:
\tt{\{kellyl,kuehn,osthus\}@maths.bham.ac.uk}}
%
%
%
\end{document}